\documentclass[11pt,tbtags,reqno]{article}
\usepackage[margin=0.96in, top=2.2cm,bottom=2.2cm]{geometry}

\usepackage[english]{babel}
\usepackage[T1]{fontenc}
\usepackage{setspace}
\usepackage{enumitem}
\usepackage[normalem]{ulem}
\usepackage{float}
\usepackage{scrextend}
\deffootnote{0em}{1.6em}{\thefootnotemark.\enskip}
\usepackage{xcolor}
\usepackage{latexsym}

\usepackage{amsmath}
\usepackage[makeroom]{cancel}
\usepackage{amssymb}
\usepackage{amsthm}
\usepackage{mathtools}
\usepackage{euscript,mathrsfs}
\usepackage{bbm}
\usepackage[colorinlistoftodos]{todonotes}
\PassOptionsToPackage{hyphens}{url}
\usepackage[colorlinks=true, allcolors=blue]{hyperref}
\usepackage{titlesec}
\usepackage[title]{appendix}
\titlelabel{\thetitle.\quad}

\mathtoolsset{showonlyrefs}

\hypersetup{
    colorlinks,
    linkcolor={blue},
    citecolor={blue},
    urlcolor={blue}
}
\usepackage{yfonts}
\usepackage{indentfirst}
\usepackage[parfill]{parskip}
\usepackage[
    citestyle=numeric,
    bibstyle=authoryear,
    dashed=false]{biblatex}
\makeatletter
\input{numeric.bbx}
\makeatother
\renewbibmacro{in:}{}
\DeclareNameAlias{author}{last-first}

\DeclareFieldFormat*{title}{#1}
\DeclareFieldFormat*[book]{title}{\textit{#1}}
\DeclareFieldFormat*{volume}{\textbf{#1}}

\newtheorem{Th}{Theorem}[section]

\newtheorem{Prop}[Th]{Proposition}

\theoremstyle{definition}
\newtheorem{remark}[Th]{Remark}

\numberwithin{equation}{section}

\theoremstyle{definition}

\newtheorem*{proof_prop_2}{Proof of Proposition \ref{stopping_time_expectation}}

\newcommand{\pr}{\mathbb{P}}
\newcommand{\ex}{\mathbb{E}}
\newcommand{\eps}{\varepsilon}

\title{\textsc{Drift Control with Discretionary Stopping for a Diffusion}}
\author{
\textsc{Václav E. Beneš} 
\thanks{ 
\,\textsc{26 Taylor Street, Millburn, NJ 07041} (e-mail: {\it beneslav@gmail.com})}
\and
\textsc{Georgy Gaitsgori} 
\thanks{ 
\,\textsc{Columbia University, Department of Mathematics, 2990 Broadway, New York, NY 10027, USA} (e-mail: {\it gg2793@columbia.edu})}
\and
\textsc{Ioannis Karatzas}
\thanks{ 
\,\textsc{Columbia University, Department of Mathematics, 2990 Broadway, New York, NY 10027, USA} (e-mail: {\it ik1@columbia.edu})}
}
\date{\today}

\begin{document}
\maketitle

\begin{abstract}
    We consider stochastic control with discretionary stopping for the drift of a diffusion process over an infinite time horizon. The objective is to choose a control process and a stopping time to minimize the expectation of a convex terminal cost in the presence of a fixed operating cost and a control-dependent running cost per unit of elapsed time. Under appropriate conditions on the coefficients of the controlled diffusion, an optimal pair of control and stopping rules is shown to exist. Moreover, under the same assumptions, it is shown that the optimal control is a constant which can be computed fairly explicitly; and that it is optimal to stop the first time an appropriate interval is visited. 
    We consider also a constrained version of the above problem, in which an upper bound on the expectation of available stopping times is imposed; we show that this constrained problem can be reduced to an unconstrained problem with some appropriate change of parameters and, as a result, solved by similar arguments.
\end{abstract}



\section{Introduction}
This paper solves explicitly a problem of stochastic control for the drift of a diffusion process, in which the controller is also allowed to stop the process and quit at any stopping time of his choice. In this way, the problem under consideration falls into the class of the so-called problems of \textit{control with discretionary stopping}, or of \textit{‘leavable control’} in the terminology of Dubins \& Savage \cite{DubSav65}, who were apparently the first to consider such problems in a formal mathematical setup.

The work of Dubins and Savage was then taken up by numerous authors, who explored problems of control with discretionary stopping from both theoretical and practical perspectives. For an in-depth understanding of the theoretical results in this area and for its historical development, we mention the works of Krylov \cite{Krylov80}, El Karoui \cite{ElKaroui81}, Bensoussan \& Lions \cite{BenLio82}. Additionally, valuable insights can be obtained from Morimoto \cite{Morimoto03}, Ceci \& Bassan \cite{CeciBas04}, Karatzas \& Zamfirescu \cite{KarZamf06}, and more recently, De Angelis \& Milazzo \cite{DeAngMil23}.
 
From the point of view of application, problems that combine control and stopping emerge in several contexts. For example, such problems arise naturally in target-tracking, when one is using control in order to steer a system close to a target but also has to decide when ‘close is close enough,’ engage the target, and ‘leave’ to avoid unnecessary costs. See the papers by Beneš \cite{Benes92}, who provides explicit solutions to linear–quadratic–gaussian problems when stopping is allowed, or Karatzas et al. \cite{KarOcoWanZer}, who discuss finite-fuel singular control problems with discretionary stopping.

Furthermore, such problems play an important role in Mathematical Finance, in particular, in the theory of option pricing under portfolio constraints. See the works by Karatzas \& Kou \cite{KarKou98}, Karatzas \& Wang \cite{KarWang00}, \cite{KarWang01}, Henderson \& Hobson \cite{HenHob08}, and Leung \& Sircar \cite{LeuSir09}, among others. We also highlight a close relationship of such problems to the intriguing concept of dynamically consistent utilities, introduced by Musiela \& Zariphopoulou \cite{MusZar07}, \cite{MusZar09}, \cite{MusZar10} (see also Berrier et al. \cite{Ber09}).

Additionally, mixed problems of control and stopping constitute a significant area of stochastic games of the principal/agent type, where one player oversees the system and the other player determines the termination time. These games involve players competing against each other, and, in this, they differ from the aforementioned problems. The development of the theory in this direction started with the paper \cite{MaiSud96p} and the book \cite{MaiSud96b} by Maitra \& Sudderth, in which a discrete-time setting was considered, and then was taken up by Kamizono \& Morimoto \cite{KamMor02}, who study variational inequalities for such games, Karatzas \& Zamfirescu \cite{KarZamf08}, who use martingale methods to prove existence results in this context, Bayraktar \& Huang \cite{BayHua12}, who show that the value function of such games is the unique viscosity solution to an appropriate Hamilton–Jacobi–Bellman equation, and others. Specific instances of these games are detailed in Karatzas \& Sudderth \cite{KarSud06}, \cite{KarSud01}, Weerasinghe \cite{Weer06}, Hernandez-Hernandez et al. \cite{HerSimZer15}, or more recently Dumitrescu et al. \cite{DumLeuTan21} in the context of Mean Field Games.

Problems that involve control with discretionary stopping are especially interesting when their solutions can be characterized explicitly. However, there are only very few results of this type. To mention some, Davis \& Zervos \cite{DavZer} consider and solve explicitly the infinite-fuel singular control problem of tracking Brownian motion with quadratic running and terminal costs, Karatzas \& Sudderth \cite{KarSud99} solve a problem in which one controls drift and variance coefficients of a diffusion process on an interval and aims to stop the process as to minimize an arbitrary continuous terminal cost. Karatzas \& Ocone \cite{KarOco02} solve explicitly a bounded-velocity control of a Brownian motion with discretionary stopping, while Ocone \& Weerasinghe \cite{OcoWeer08} solve a variance control problem for a diffusion process with linear drift, in which the variance is allowed to vanish.

\noindent
\underline{\textbf{Preview:}}
In this paper, we solve explicitly another problem of this type; namely, we consider stochastic control with discretionary stopping for the drift of a diffusion process, over an infinite time horizon. More specifically, we consider controlled diffusion processes $X(\cdot)$ on a filtered probability space $(\Omega, \mathcal{F}, \pr)$, $\mathbb{F} =\nobreak \{\mathcal{F}(t)\}_{t \ge 0}$, which satisfy
\begin{equation}
    X(t) = x + \int_{0}^{t} u(r)\mu(X(r)) \,dr + \int_{0}^{t} \sigma (X(r)) \, dW(r), \quad 0 \le t < \infty.
\end{equation}
Here $x \in \mathbb{R}$ is a given starting position, $W(\cdot)$ is an $\mathbb{F}$--Brownian motion, $u(\cdot)$ is an $\mathbb{F}$--progressively-measurable process, and $\mu: \mathbb{R} \to \mathbb{R}, \, \sigma: \mathbb{R} \to \mathbb{R}$ are given measurable functions. 

Our goal is to find a control process $u(\cdot)$ and a stopping time $\tau$, adapted to the filtration generated by the process $X(\cdot)$, so as to minimize the total expected cost
\begin{equation}\label{objective_to_mininmize}
    \mathbb{E}\left[k(X(\tau))+\int_{0}^{\tau} \psi(u(t)) \, dt + c \tau\right],
\end{equation}
where $c > 0$ is an ``operating'' cost per unit of time, $k(\cdot)$ is a ``terminal cost'', and $\psi(\cdot)$ is a “running cost of control”.

It turns out that, under appropriate assumptions, the above problem has a very simple and fairly explicit solution; and this is despite the fact that it is posed on an infinite time horizon and without any discounting. Specifically, if the equation \eqref{equation_for_stopping_point} below, which involves the cost functions $\psi(\cdot)$, $k(\cdot)$, and the parameters $c$ and $A$,
has a solution in the positive half-line, then the optimal control process $u^*(\cdot)$ is given by a suitable constant. The optimal stopping time $\tau^*$, in this case, is the first entrance time in some appropriate interval $[-s, s]$ (where the constant $s \ge 0$ is also computable fairly explicitly) by the optimal process $X^*(\cdot)$, i.e., 
\begin{equation}
    \tau^* = \inf \{ t \ge 0: |X^*(t)| \le s\}.
\end{equation}
If, on the contrary, the equation \eqref{equation_for_stopping_point} 
has no solution in the positive half-line, then the best thing to do is to \textit{stop at once}, i.e., select $\tau^* \equiv 0$; in this case, no control is ever exerted. 

In the last part of the paper, we consider a constrained version of the above problem. More specifically, we study the same controlled diffusion processes and again minimize the expected cost \eqref{objective_to_mininmize}. However, this time we minimize only over stopping times $\tau$ which satisfy $\ex[\tau] \le \alpha$, where $\alpha \ge 0$ is some fixed parameter. Using duality arguments, we show that such a constrained problem is actually equivalent to an unconstrained problem of the same type as the initial problem \eqref{objective_to_mininmize}, but with an appropriate change of parameters $c \mapsto c + \lambda$, where the constant $\lambda \ge 0$ is to be determined. As a result, the constrained problem has essentially the same solution as an unconstrained one, but with a possibly shifted stopping level $s$ and a scaled optimal control $u^*(\cdot)$.

The structure of the paper is as follows. We formalize our setup in Section \ref{section_model} and prove our main results in Section \ref{section_main_results}. Section \ref{section_constrained_problem} is devoted to the constrained problem. Another possible extension is discussed in the Appendix. We also mention some open problems in Section \ref{section_open_questions}.

\section{Model}\label{section_model}

As already mentioned, we consider a controlled diffusion process $X(\cdot)$ on a filtered probability space $(\Omega, \mathcal{F}, \pr)$, $\mathbb{F} =\nobreak \{\mathcal{F}(t)\}_{t \ge 0}$, which satisfies
\begin{equation}\label{controlled_sde}
    X(t) = x + \int_{0}^{t} u(r)\mu(X(r)) \,dr + \int_{0}^{t} \sigma (X(r)) \, dW(r), \quad 0 \le t < \infty.
\end{equation}
Here $x \in \mathbb{R}$ is a given starting position, $W(\cdot)$ is an $\mathbb{F}$--Brownian motion, $u(\cdot)$ is an $\mathbb{F}$--progressively-measurable process, and $\mu: \mathbb{R} \to \mathbb{R}, \, \sigma: \mathbb{R} \to \mathbb{R}$ are given measurable functions which satisfy the following assumptions:
\begin{enumerate}[label = \textbf{(A\arabic*)}]
    \item $\mu(\cdot)$ is even, and positive and differentiable on $(0, \infty)$, \label{ass_1}
    \item $\sigma(\cdot)$ is even, continuous, and bounded away from zero on $(\eps, \infty)$ for every $\eps > 0$, \label{ass_2}
    \item $\int_{0}^{t}\big[|u(r)\, \mu(X(r))|+\sigma^{2}(X(r))\big] d r < \infty$ holds almost surely, for every $0 \leq t < \infty$, \label{ass_3}
    \item $\mu'(x) \, \sigma^2(x) = A \, \mu^2(x)$ for some $A \ge 0$ and all $x \in [0, \infty)$, \label{ass_4}
    \item If $A = 0$ in \ref{ass_4}, then the function $\sigma(\cdot)$ is bounded.
    If $A > 0$ in \ref{ass_4}, then the function $\mu'(\cdot)$ is bounded away from zero on $(\varepsilon, \infty)$ for every $\varepsilon > 0$. \label{unf_assum}
\end{enumerate}
\begin{remark}
    The assumption that the functions $\mu(\cdot)$ and $\sigma(\cdot)$ are symmetric is made just for convenience of exposition and can be relaxed. The third assumption guarantees that the Lebesgue and Itô integrals in \eqref{controlled_sde} are well-defined. 
    
    \vspace{-7pt}
    The fourth assumption is an unusual and quite restrictive relation between the coefficients of the controlled diffusion \eqref{controlled_sde}. However, many special cases of interest do satisfy \ref{ass_4}. For instance, all processes with the constant drift term $\mu(\cdot) \equiv const$, or a geometric Brownian motion, for which $\mu(x) = \sigma(x) = x$ for all $x \in \mathbb{R}$. 

    \vspace{-7pt}
    The fifth assumption provides growth conditions on the functions $\mu(\cdot)$ and $\sigma(\cdot)$, which allow our arguments to work; it is conceivable that it may be relaxed.
\end{remark}

For a given initial position $x \in \mathbb{R}$, denote by $\mathcal{A}(x)$ the collection of all ``admissible processes'' $X(\cdot)$ that can be constructed this way. By this we mean that each element of $\mathcal{A}(x)$ is a quintuple consisting of a probability space $(\Omega, \mathcal{F}, \pr)$, a filtration $\mathbb{F} =\nobreak \{\mathcal{F}(t)\}_{t \ge 0}$ of sub-$\sigma$-algebras of $\mathcal{F}$, an $\mathbb{F}$--adapted Brownian motion $W(\cdot)$, an $\mathbb{F}$--progressively-measurable process $u(\cdot)$, and an $\mathbb{F}$--adapted process $X(\cdot)$ which satisfies \eqref{controlled_sde}. Note that, for every given $x \in \mathbb{R}$, the collection $\mathcal{A}(x)$ is not empty, since the assumption \ref{ass_2} implies the existence of a weak solution of the equation \eqref{controlled_sde} in case one sets $u(t) \equiv 0$ for all $t \ge 0$ (see, e.g., Theorems 5.4, 5.7, pp. 332, 335 in \cite{BMSC}). We will abuse notation and write $X(\cdot) \in \mathcal{A}(x)$, but keep always in mind that $X(\cdot)$ is part of a quintuple $\left((\Omega, \mathcal{F}, \pr), \mathbb{F}, W(\cdot), u(\cdot), X(\cdot) \right)$, also called ``weak solution'' of the equation \eqref{controlled_sde}, for some $\mathbb{F}$--progressively-measurable control process $u(\cdot)$ and some driving Brownian motion $W(\cdot)$ adapted to $\mathbb{F}$ as well. 

Now, for each element in $\mathcal{A}(x)$ with corresponding so-called ``state process'' $X(\cdot)$, let $\mathbb{F}^X \coloneqq \{\mathcal{F}^X(t), 0 \le t < \infty \}$ be the filtration generated by $X(\cdot)$, where $\mathcal{F}^X(t) \coloneqq \sigma(X(s), 0 \leq s \le t)$, and denote the collection of all $\mathbb{F}^X$--stopping times by $\mathcal{T}^X$.

With the above notation, the controller solves the following optimization problem. \textit{For a given initial position $x \in \mathbb{R}$, find an element $X^*(\cdot) \in \mathcal{A}(x)$ and an associated  stopping time $\tau^* \in \mathcal{T}^{X^*}$, such that the pair $(X^*(\cdot), \, \tau^*)$ minimizes the total expected cost}
\begin{equation}\label{total_cost}
    J(x, X(\cdot), \tau) \coloneqq \mathbb{E}\left[k(X(\tau))+\int_{0}^{\tau} \psi(u(t)) \, dt + c \tau\right]
\end{equation}
over all $X(\cdot) \in \mathcal{A}(x)$ and $\tau \in \mathcal{T}^X$. 
Here $c > 0$ is an ``operating'' cost per unit of time, $k(\cdot)$ a ``terminal cost'', and $\psi(\cdot)$ a “running cost of control.” We impose the following assumptions on these functions:
\begin{enumerate}[label = \textbf{(A\arabic*)}]
\setcounter{enumi}{5}
    \item both $k(\cdot)$ and $\psi(\cdot)$ are of class $C^2$, non-negative, even, and strictly convex with $\psi(0) = 0$, \label{ass_5}
    \item $\lim\limits_{x \to \infty} \psi'(x) = \infty$. \label{ass_6}
\end{enumerate}

Finally, we denote the value function of this problem by
\begin{equation}\label{value_function}
    V(x) \coloneqq \inf\limits_{\substack{X(\cdot) \in \mathcal{A}(x) \\ \tau \in \mathcal{T}^X}} J(x, X(\cdot), \tau).
\end{equation}
Note that, despite formally minimizing over all stopping times $\tau \in \mathcal{T}^X$, quite clearly, we need only consider stopping times $\tau \in \mathcal{T}^X$ such that $\ex[\tau] < \infty$, because $c > 0$.

\begin{remark}
    Of course, when $c = 0$, we can take $u^*(\cdot) \equiv 0$ and just wait until the first time $\tau_\eps$ the diffusion process $X^*(\cdot) = x + \int_0^\cdot \sigma(X^*(t)) \, dW(t)$ hits an arbitrarily small level $\eps > 0$ (which will happen eventually, i.e., $\pr(\tau_\eps < \infty) = 1$, because $\sigma(\cdot)$ is bounded away from zero on $[\eps, \infty)$ by the assumption \ref{ass_2}). This leads to $J(x, X^*(\cdot), \tau_\eps) = k(\eps)$ and consequently $V(x) = k(0)$ by the arbitrariness of $\eps$, the lowest achievable expected cost, since the cost function $\psi(\cdot)$ is non-negative with $\psi(0) = 0$, while $k(\cdot)$ attains its minimum at the origin.
\end{remark}

\begin{remark}\label{remark_wlog_positive}
    It is clear that, for any pair $(X(\cdot), \tau)$ and a stopping time $\tau_0 \coloneqq \inf\{t: X(t) = 0\}$, we have  $J(x, X(\cdot), \tau \wedge \tau_0) \le J(x, X(\cdot), \tau)$; i.e., it is never optimal to proceed once the origin has been reached. This follows from the fact that the terminal cost function $k(\cdot)$ attains its minimum at the origin, while the cost of control $\psi(\cdot)$ and the operating cost $c$ are non-negative. As a result,
    without loss of generality, our problem can be studied only on the positive half-line, and the complementary case can be obtained by symmetry.
\end{remark}

\begin{remark}
    Even though we minimize in \eqref{value_function} over stopping times $\tau \in \mathcal{T}^X$, it will become rather obvious to the reader that our results and the arguments of the next section still hold when the infimum in \eqref{value_function} is taken instead over all stopping times $\tau \in \mathcal{T}$, where $\mathcal{T}$ denotes the collection of \underline{all} stopping times of the filtration $\mathbb{F}$ in the quintuple $\left((\Omega, \mathcal{F}, \pr), \mathbb{F}, W(\cdot), u(\cdot), X(\cdot) \right)$. 
\end{remark} 

It is clear from the form of the criterion \eqref{total_cost} and the properties \ref{ass_5}, \ref{ass_6} of its cost functions, that it should make sense to stop immediately if starting sufficiently near the origin; but that it might also make sense to exert some wise control for a while if the starting position is way far from, and then stop the moment we have arrived sufficiently near, the origin. 

We are now ready to analyze the problem and substantiate the above intuition.

\section{Main results}\label{section_main_results}

To solve a stochastic control problem, one must compute its value function and provide a controlled process that attains this value. Therefore, we shall proceed as follows. First, we introduce variational inequalities, which, as we conjecture, the value function should satisfy. Secondly, using the assumptions \ref{ass_1}--\ref{ass_6}, we find a function that satisfies these inequalities and hence becomes a candidate for the value function of the problem. 

\vspace{-7pt}
The form of the candidate value function will immediately suggest candidates for the optimal control process and stopping time. Thus, our last step will be to show that the candidate value function is indeed the actual value function, i.e., that it provides a lower bound on the achievable cost in the optimization problem \eqref{total_cost}, and that the suggested controlled process and stopping time actually attain this lower bound.

\subsection{Variational inequalities}\label{subsec_var_inequalities}

Variational inequalities are common tools in problems of stochastic control with discretionary stopping (see, e.g., \cite{DavZer} or \cite{KarOcoWanZer}). Accordingly, we conjecture, and it will become clear from the proof of the main theorem below, that the value function $V(\cdot)$ should satisfy the following variational inequalities for all $x \in \mathbb{R}$:
\begin{enumerate}[label = \textbf{(\roman*)}]
    \item $k(x) - V(x) \ge 0$; \label{var_ineq_1}
    \item $\frac{1}{2}V''(x)\sigma^2(x) + \min\limits_{u \in \mathbb{R}}\Big[ u \,\mu(x)\, V'(x) + \psi(u) \Big] + c \ge 0$; \label{var_ineq_2}
    \item $\Big( k(x) - V(x)\Big)\left( \frac{1}{2}V''(x)\sigma^2(x) + \min\limits_{u \in \mathbb{R}}\Big[ u \,\mu(x)\, V'(x) + \psi(u) \Big] + c
    \right) = 0$. \label{var_ineq_3}
\end{enumerate}
Here, the first inequality follows clearly from the fact that we can stop the diffusion immediately, i.e., take $\tau \equiv 0$, and for such $\tau$ and any controlled process $X(\cdot)$ we get $J(x, X(\cdot), \tau) = k(x)$. The second is a Bellman--type inequality describing the optimization of the controlled drift. The third equality tells us that, for an optimal pair consisting of a controlled process and a stopping time to exist, we must at all times either control optimally or stop the process. 

The second inequality \ref{var_ineq_2} can be simplified as follows. Without loss of generality, consider the case $x \ge 0$ and observe that, for fixed $x$, the function $u \mapsto u \,\mu(x)\, V'(x) + \psi(u)$ is minimized at 
\begin{equation}\label{optimal_control_1}
    u = \xi\big(-\mu(x) \, V'(x)\big),
\end{equation}
where 
\begin{equation}\label{def_inversi_psi_prime}
    \xi(\cdot) \coloneqq \left(\psi'\right)^{-1}(\cdot)
\end{equation}
is the inverse of the strictly increasing function $\psi'(\cdot)$, whose existence is guaranteed by the assumptions \ref{ass_5}, \ref{ass_6}. Therefore, substituting for $u$ the quantity of \eqref{optimal_control_1} back into \ref{var_ineq_2}, we can rewrite this variational inequality as
\begin{equation}\label{variational_inequality_with_lambda}
     \frac{1}{2}V''(x)\sigma^2(x) + \eta\big(\mu(x) \, V'(x)\big) + c \ge 0,
\end{equation}
with the Legendre-type
\begin{equation}\label{eta_function}
    \eta(z) \coloneqq \min\limits_{u \in \mathbb{R}}\Big[ uz + \psi(u) \Big] = z \, \xi(-z) + \psi(\xi(-z))
\end{equation}
transform of $\psi(\cdot)$. We are now ready to solve the above variational inequalities. 

\subsection{Candidate value function}

For general cost functions $\psi(\cdot)$, the equation \eqref{variational_inequality_with_lambda} is a non-linear second-order differential equation. Thus, one wouldn't expect to be able to solve it explicitly and obtain an explicit form for the value function of the control problem. However, the key assumption \ref{ass_4} on the relation between $\mu(\cdot)$ and $\sigma(\cdot)$ allows us to get around this non-linearity by the following reasoning.

Let us focus attention on the case $x \ge 0$, since the case $x < 0$ can be dealt with similarly by the symmetry of the problem. Observe that if, for some $\gamma \in \mathbb{R}$, we have 
\begin{equation}\label{equation_on_clope}
    \eta(\gamma) - \frac{1}{2}A\gamma + c = 0,
\end{equation}
where $A$ is the constant from the assumption \ref{ass_4} and $\eta(\cdot)$ is the function in \eqref{eta_function}, then the same assumption implies that, for fixed $s > 0, \, b \in \mathbb{R}$, the function 
\begin{equation}\label{function_for_second_var_ineq}
    f(x) \coloneqq \gamma \int\limits_s^x \frac{dy}{\mu(y)} + b
\end{equation}
satisfies $\mu(x) f'(x) = \gamma$, as well as the equation
\begin{equation}
    \frac{1}{2}f''(x)\sigma^2(x) + \eta(\mu(x) \, f'(x)) + c = 0
\end{equation}
on the half-line $[s, \infty)$.
In other words, the function $f(\cdot)$ in \eqref{function_for_second_var_ineq} satisfies the variational inequality \ref{var_ineq_2} in the form \eqref{variational_inequality_with_lambda}, as well as the variational equality \ref{var_ineq_3} on $[s, \infty)$. 

At the same time, we note that the choice $V(\cdot) = k(\cdot)$ satisfies the inequality \ref{var_ineq_1} and the variational equality \ref{var_ineq_3} in the complementary region $[0, s)$. As a result of these considerations, we put together the following candidate for value function:
\begin{equation}\label{value_function_candidate}
    V(x) = 
    \begin{cases}
        k(x),  & 0 \leq x < s, \\
        \gamma \int_s^x \frac{dy}{\mu(y)} + b, & x \ge s, \\
        V(-x),  & x<0,
    \end{cases}
\end{equation}
where the real constants $s > 0, b$ and $\gamma$ need to be specified. 

To find these constants, we use the heuristic principle of so-called ``smooth fit'', which mandates that the value function $V(\cdot)$ should be of class $C^1$. We comply with this requirement, and also with the equation \eqref{equation_on_clope}, by requiring that the real constants $s > 0, b$ and $\gamma$ obey the following system of equations:
\begin{equation}
    \begin{cases}
        k(s) = b, \\
        k'(s) = \gamma / \mu(s), \\
        \eta(\gamma) - \frac{1} {2}A\gamma + c = 0.
    \end{cases}
\end{equation}
Therefore, let $s>0$ be the solution to the scalar equation
\begin{equation}\label{equation_for_stopping_point}
    \eta(\mu(s) \, k'(s)) -\frac{1}{2}A\, \mu(s) \, k'(s) + c = 0,
\end{equation}
which imitates \eqref{equation_on_clope} with $\gamma = \mu(s) \, k'(s)$ according to the second of the equations right above, and set $s \coloneqq \infty$ if the equation \eqref{equation_for_stopping_point} doesn't have a solution in $(0, \infty)$. 

Note that the equation \eqref{equation_for_stopping_point} has at most one positive solution. This is because, on $(0, \infty)$, the function $\mu(\cdot) \, k'(\cdot)$ is strictly increasing, as can be verified easily via \ref{ass_1}, \ref{ass_4}--\ref{ass_5}; while the function $\eta(\cdot)$ of \eqref{eta_function} is strictly decreasing, as can shown by observing from \eqref{eta_function} that $\eta(\cdot)$ satisfies 
\begin{equation}\label{properties_of_eta}
    \eta'(z) = \xi(-z) < 0 \quad \text{ for all } z > 0.
\end{equation}
In particular, the function 
\begin{equation}\label{def_zeta_function}
    \zeta(z) \coloneqq \eta(\mu(z) \, k'(z)) -\frac{1}{2}A\, \mu(z) \, k'(z),
\end{equation}
is strictly decreasing, so the solution $s = \zeta^{-1}(-c) \eqqcolon s_c$ of the scalar equation \eqref{equation_for_stopping_point} increases as $c > 0$ increases. This makes good intuitive sense: the higher the ``operating'' cost $c$ per unit of elapsed time, the bigger the stopping threshold level $s$.

Putting $b = k(s)$ and $\gamma = \mu(s) \, k'(s)$, we rewrite \eqref{value_function_candidate} in the form
\begin{equation}\label{value_function_solution}
    V(x) = 
    \begin{cases}
        k(x),  & 0 \leq x < s, \\
        k(s) + \mu(s) \, k'(s) \int_s^x \frac{dy}{\mu(y)}, & x \ge s, \\
        V(-x),  & x<0,
    \end{cases}
\end{equation}
and claim that the so-defined function $V(\cdot)$ is indeed the value function of the control problem under consideration. 
Note that we set $s = \infty$ if the equation \eqref{equation_for_stopping_point} doesn't have a solution in $(0, \infty)$, thus the function $V(\cdot)$ in \eqref{value_function_solution} becomes $V(\cdot) \equiv k(\cdot)$ in this case.

Before stating our main result, which justifies the above claim, let us verify first that the function $V(\cdot)$ defined in \eqref{value_function_solution} indeed solves the variational inequalities \ref{var_ineq_1}--\ref{var_ineq_3} of \S \ref{subsec_var_inequalities}. Inequality \ref{var_ineq_1} is obviously satisfied (as equality) on $[-s, s]$, and is satisfied on $(-\infty, -s) \cup (s, \infty)$ because for all $x \ge s$ we have
$$
k'(x) \ge k'(s) \ge \mu(s) k'(s) \frac{1}{\mu(x)} = \left(k(s)+ \mu(s) \, k'(s) \int_s^x \frac{dy}{\mu(y)}\right)' = V'(x),
$$
where the first inequality is due to the convexity of $k(\cdot)$, and the second inequality is due to monotonicity of $\mu(\cdot)$ on $(0, \infty)$. 
Inequality \ref{var_ineq_2} is obviously satisfied on $(-\infty, -s) \cup (s, \infty)$ by construction. 

(a) If $s < \infty$, then on $[-s, s]$ the inequality \ref{var_ineq_2} follows from the chain of inequalities
\begin{equation}
\begin{split}
    \frac{1}{2}k''(x)\sigma^2(x) + \min\limits_{u \in \mathbb{R}}\Big[ u \,\mu(x)\, k'(x) + \psi(u) \Big] + c &=
    \frac{1}{2}k''(x)\sigma^2(x) + \eta(\mu(x)\, k'(x)) + c \\
    &\ge
    \eta(\mu(x)\, k'(x)) + c \\
    &\ge
    \eta(\mu(s)\, k'(s)) + c \\
    &= \frac{1}{2}A\, \mu(s)\, k'(s) \\
    &\ge 0.
\end{split}
\end{equation}
Here, the first inequality follows from the convexity of $k(\cdot)$, the second inequality follows from the fact that  $\mu(\cdot)\, k'(\cdot)$ is an increasing function on $(0, \infty)$, while $\eta(\cdot)$ is a decreasing function on $(0, \infty)$ on account of \eqref{properties_of_eta}. The equality on the fourth line follows by the choice of $s$ and \eqref{equation_for_stopping_point}, and the last inequality follows by the monotonicity of $k(\cdot)$ and the non-negativity of $A$ and $\mu(\cdot)$. 

(b) If, on the other hand, $s = \infty$, i.e., the equation \eqref{equation_for_stopping_point} does not have a positive solution, then the same arguments can be applied with the only correction that $\eta(\mu(x) \, k'(x)) + c \ge 0$ holds now for all $x \in \mathbb{R}$ by the non-existence of a solution to \eqref{equation_for_stopping_point}. 

Finally, the inequality \ref{var_ineq_3} is satisfied by the choice of $s$ and \eqref{equation_for_stopping_point}. 

\subsection{Main Theorem}

We are now ready to state our main result. 

\begin{Th}\label{main_result}
    With the above assumptions and notation, the value function $V(\cdot)$ for the stochastic control problem of \eqref{value_function} with discretionary stopping is given by the expression of \eqref{value_function_solution}. 
    
    Moreover, the infimum in the definition \eqref{value_function} of the value function is attained, and there are two cases:
    
    \underline{Case 1.} If the equation \eqref{equation_for_stopping_point} has no positive solution, then the stopping time $\tau^* \equiv 0$ is optimal, and no control is ever exerted. 

    \underline{Case 2.} If the equation \eqref{equation_for_stopping_point} has a solution $s > 0$, then this solution is unique and the constant control
    \begin{equation}\label{optimal_control}
        u^*(t) \equiv Q(s), \quad \forall \,\, 0 \le t < \infty \quad \quad \text{with} \quad \quad 
        Q(s) \coloneqq \left(\psi'\right)^{-1}\big(-\mu(s) k'(s)\big)
    \end{equation}
    is optimal; and so is the stopping time
    \begin{equation}\label{optimal_stopping_time}
        \tau^* \coloneqq \inf \{t \geq 0: |X^*(t)| \leq s\},
    \end{equation}
    i.e., the first entrance time in the interval $[-s, s]$ by the process $X^*(\cdot)$.
    
    Moreover, a weak solution to the equation \eqref{controlled_sde} with the given control process of \eqref{optimal_control} and stopping time of \eqref{optimal_stopping_time} exists.
\end{Th}

\begin{proof}
We first show that the function $V(\cdot)$ given by \eqref{value_function_solution} is a lower bound on the achievable cost, and then use the pairs described in Theorem \ref{main_result} to construct the process and the stopping time, which attain this lower bound. 

\noindent
\textbf{I. Verification for the Value Function}

Consider any process $X(\cdot)$ that can be constructed on some filtered probability space $(\Omega, \mathcal{F}, \pr)$, $\mathbb{F} \nobreak= \{\mathcal{F}(t)\}_{t \ge 0}$, rich enough to support also an $\mathbb{F}$--Brownian motion $W(\cdot)$, so as to satisfy \eqref{controlled_sde} for some $\mathbb{F}$--progressively measurable control process $u(\cdot)$. 
Following the arguments of Remark \ref{remark_wlog_positive}, we assume, without loss of generality, that $x \ge 0$, and let $\tau$ be any $\mathbb{F}^X$--stopping time such that $\tau \le \tau_0 \coloneqq \inf\{t: X(t) = 0\}$ almost surely. In particular, we note that the process $X(\cdot \wedge \tau)$ is almost surely non-negative. Moreover, we can assume that $\ex[\tau] < \infty$, since otherwise the expected cost \eqref{total_cost} is automatically infinite and such a stopping time cannot be optimal, because the cost $k(x)$ associated with immediate stopping $\tau \equiv 0$ is finite.

Consider now the process $V(X(\cdot))$. Since $X(\cdot)$ is an $\mathbb{F}$--semi-martingale and the function $V(\cdot)$ is almost everywhere $C^2$ (except possibly at $\pm s$), in particular, has derivative absolutely continuous with respect to the Lebesgue measure, we have, on the strength of the  Itô's rule in Problem 7.3, p. 219 of \cite{BMSC}:
\begin{equation}\label{Ito}
    \begin{split}
        V(X(\tau)) - V(x) 
        &= 
        \int\limits_{0}^{\tau} V'(X(t)) \, dX(t) + \frac{1}{2}\int\limits_{0}^{\tau} V''(X(t))\, \sigma^2(X(t))\, dt \\
        &=
        \int\limits_{0}^{\tau} V'(X(t))\, \sigma(X(t)) \, dW(t) \\ 
        & \quad \quad \quad \quad + 
        \int\limits_{0}^{\tau} \left\{  V'(X(t))\, u(t)\mu(X(t))+ \frac{1}{2} V''(X(t))\, \sigma^2(X(t))\right\}\, dt.
        \end{split}
\end{equation}
To proceed further, we consider separately the local martingale 
\begin{equation}
    M(t) \coloneqq \int\limits_{0}^{\tau \wedge t} V'(X(r))\, \sigma(X(r)) \, dW(r), \quad 0 \le t < \infty,
\end{equation}
and show that it is in fact a bounded in $\mathbb{L}^2$ martingale. To do that, we note that it has quadratic variation 
\begin{equation}\label{stoch_integral_quadratic_variation}
    \langle M \rangle (t)=
        \int\limits_{0}^{\tau \wedge t} \Big(V'(X(r))\, \sigma(X(r))\Big)^2 \, dr,
\end{equation}
which is integrable at infinity, i.e., $\mathbb{E}[\langle M \rangle(\infty)] < \infty$. The last statement follows from the boundedness of the integrand in the expression \eqref{stoch_integral_quadratic_variation}, which can be shown by the following arguments.

We consider the events $\{X(r) < s\}$ and $\{X(r) \ge s\}$ separately. On $\{X(r) < s\}$, we have the upper bound
\begin{equation}\label{first_upper_bound}
    \Big(V'(X(r))\, \sigma(X(r))\Big)^2 \le \Big(k'(s) \cdot \max\limits_{0 \le x \le s} \sigma(x)\Big)^2 < \infty
\end{equation}
due to the convexity of $k(\cdot)$ and the continuity of $\sigma(\cdot)$. 

Whereas on $\{X(r) \ge s\}$, we first write 
\begin{equation}
    \Big(V'(X(r))\, \sigma(X(r))\Big)^2 =
        \big(\mu(s) k'(s)\big)^2 \, \frac{\sigma^2(X(r))}{\mu^2(X(r))},
\end{equation}
and then, in the case $A > 0$, obtain the upper bound
\begin{equation}\label{second_upper_bound_a}
    \begin{split}
        \big(\mu(s) k'(s)\big)^2 \, \frac{\sigma^2(X(r))}{\mu^2(X(r))} 
        &= 
        \big(\mu(s) k'(s)\big)^2 \, \frac{A}{\mu'(X(r))}
        \\ &\le
        \big(\mu(s) k'(s)\big)^2 \, \frac{A}{\min\limits_{s \le q < \infty} \mu'(q)} < \infty.
    \end{split}
\end{equation}
Here the equality follows from the assumption \ref{ass_4}, and the second inequality from the assumption \ref{unf_assum}. Whereas, in the case $A = 0$, we note that the assumption \ref{ass_4} implies $\mu'(x) = 0$ for all $x > 0$, and, consequently, for all $x > 0$ we have $\mu(x) \equiv \mu$ for some positive real constant $\mu$, leading to the upper bound
\begin{equation}\label{second_upper_bound_b}
     \big(\mu(s) k'(s)\big)^2 \, \frac{\sigma^2(X(r))}{\mu^2(X(r))}
     = (k'(s))^2 \cdot \Sigma < \infty,
\end{equation}
where $\Sigma > 0$ is the real upper bound on $\sigma^2(\cdot)$, courtesy of the assumption \ref{unf_assum} again.

Combining these upper bounds, we deduce that there exists a positive real constant $C$ with
\begin{equation}
    \ex [M^2(t)] \le \ex [\langle M \rangle(t)] \le \mathbb{E}[\langle M \rangle(\infty)] \le C \cdot \ex [\tau] < \infty, \quad \forall \, t \ge 0,
\end{equation}
where the first inequality is well-known for local martingales (see, e.g., p. 38 in \cite{BMSC}), and the second follows from the monotonicity of quadratic variation. As a result, we obtain that $M(\cdot)$ is in fact a bounded in $\mathbb{L}^2$ martingale.

Returning to the equation \eqref{Ito}, by taking expectations on both sides we obtain 
\begin{equation}\label{value_function_as_lower_bound}
    \begin{split}
        V(x) &= \ex \left[V(X(\tau)) - \int\limits_{0}^{\tau} \left\{  V'(X(r))u(r)\mu(X(r))+ \frac{1}{2} V''(X(r))\sigma^2(X(r))\right\}\, dr \right]\\
        &\le
        \ex \left[k(X(\tau)) + \int\limits_{0}^{\tau}  \psi(u(r)) \, dr + c\, \tau \right]
        =
        J(x, X(\cdot), \tau).
    \end{split}
\end{equation}
Here the inequality on the second line follows from the fact that $V(\cdot)$ satisfies the variational inequalities \ref{var_ineq_1}--\ref{var_ineq_2}, while the equality on the first line uses the fact that the expectation of the stochastic integral in \eqref{Ito} vanishes due to the optional sampling theorem for martingales bounded in $\mathbb{L}^2$. Therefore, $V(x)$ is indeed a lower bound on the achievable cost. 

It remains to show that the pair of control process and stopping time, described in Theorem \ref{main_result}, attains this lower bound $V(x)$ in \eqref{value_function_as_lower_bound}.

\noindent
\textbf{II. Verification of Optimality}

We note first that, if the equation \eqref{equation_for_stopping_point} has no solution on $(0, \infty)$, then the value function $V(\cdot)$ of \eqref{value_function_solution} is identically equal to $k(\cdot)$; and thus, by taking $\tau^* \equiv 0$, we obviously attain the lowest possible expected cost. 

Therefore, we assume from now onwards that \eqref{equation_for_stopping_point} has a solution $s \in (0, \infty)$. In this case, let $u^*$ be the constant control as in \eqref{optimal_control}, consider a corresponding controlled process $X^*(\cdot)$, and let the stopping time $\tau^*$ be defined as in \eqref{optimal_stopping_time}. For such a pair $(X^*(\cdot), \tau^*)$, the inequality in \eqref{value_function_as_lower_bound} becomes an equality. Indeed, by the definitions of the stopping time $\tau^*$ and the value function $V(\cdot)$, we have $k(X^*(\tau^*)) = k(s) = V(s) = V(X^*(\tau^*))$. Moreover, the process $X^*(\cdot \wedge \tau^*)$ takes values in the region $[s, \infty)$, on which the third variational inequality \ref{var_ineq_3} is actually an equality. As a result, the inequality in \eqref{value_function_as_lower_bound} also becomes equality, which implies that the pair $(X^*(\cdot), \tau^*)$ attains the lowest possible achievable cost and thus is optimal.
\hfill \qed {\parfillskip0pt\par}

\subsection{An Alternative Argument of Verification of Optimality}

Even though the above arguments complete the proof of the \textit{Verification of Optimality} step, we find the following alternative arguments for this step of considerable independent interest and, therefore, include them in the paper. Moreover, these arguments will be necessary for the subsequent treatment of a constrained version of the problem under consideration. Its proof is given in Appendix \ref{appendix_proof}.

We consider again the case when the equation \eqref{equation_for_stopping_point} has a solution $s \in (0, \infty)$. Then, the following result allows us to complete the proof.
\begin{Prop}\label{stopping_time_expectation}
    Consider a diffusion process of the form
    \begin{equation}\label{proposition_diffusion}
        \begin{split}
        dX(t) &= u\, \mu(X(t)) \,dt + 
        \sigma (X(t)) \,
        dW(t), \quad 0 \le t < \infty,\\
        X(0) &= x,
    \end{split}
    \end{equation}
    where $x \ge 0$, $u < 0$ are real constants, and the functions $\mu(\cdot)$ and $\sigma(\cdot)$ satisfy the assumptions \ref{ass_1}--\ref{unf_assum}. Then, for $\, 0 < s \le x < \infty$, the $\mathbb{F}^X$--stopping time
    \begin{equation}\label{one_sided_hitting_time}
        \tau \coloneqq \inf\{t \ge 0 : X(t) \le s \}
    \end{equation}
    satisfies \begin{equation}\label{stop_time_expectation_formula}
        \mathbb{E}[\tau] = \frac{2}{A - 2u}\int_s^x \frac{dy}{\mu(y)} < \infty.
    \end{equation}
\end{Prop}
With this result at hand, we note that with constant control $u^*$ as in \eqref{optimal_control} and stopping time $\tau^*$ as in \eqref{optimal_stopping_time}, the expected payoff, corresponding to the pair $(X^*(\cdot), \tau^*)$, is given by 
\begin{equation}\label{chain_for_J}
    \begin{split}
        J(x, X^*(\cdot), \tau^*) 
        &=
        \mathbb{E}\left[ k(X^*(\tau^*)) + \int_0^{\tau^*} \psi(u^*) \, ds + c\tau^* \right] 
        \\ &= 
        \mathbb{E}\Big[ k(s) + \tau^*(c + \psi(u^*)) \Big]
        \\ &=
        k(s) + (c + \psi(u^*)) \cdot \frac{2}{A - 2u^*}\int_s^x \frac{dy}{\mu(y)} 
        \\ &=
        k(s) + \mu(s) \, k'(s) \int_s^x \frac{dy}{\mu(y)},
    \end{split}
\end{equation}
which is exactly the expression for the value function \eqref{value_function_solution}. 

\vspace{-7pt}
In the above chain of equalities, the first and second follow from the definitions of $J(x, X^*(\cdot), \tau^*)$ in \eqref{total_cost} and $u^*(\cdot)$ in \eqref{optimal_control}, whereas the third is a consequence of \eqref{optimal_stopping_time} and Proposition \ref{stopping_time_expectation}. 

\vspace{-7pt}
The last equality in \eqref{chain_for_J} follows from the scalar equation \eqref{equation_for_stopping_point} and the definition of the optimal control \eqref{optimal_control}. Indeed, using the definitions of the functions $\eta(\cdot)$ in \eqref{eta_function}, $\xi(\cdot)$ in \eqref{def_inversi_psi_prime}, and $Q(\cdot)$ in \eqref{optimal_control}, the equation \eqref{equation_for_stopping_point} can be written as
\begin{equation}
    \mu(s) \, k'(s) \, Q(s) + \psi\left(Q(s)\right)-\frac{1}{2}A\, \mu(s) \, k'(s) + c = 0,
\end{equation}
or equivalently
\begin{equation}
    \mu(s) \, k'(s) \cdot u^* + \psi\left(u^*\right)-\frac{1}{2}A\, \mu(s) \, k'(s) + c = 0,
\end{equation}
in the notation of \eqref{optimal_control}. After rearranging, we obtain
\begin{equation}
    (c + \psi(u^*)) \cdot \frac{2}{A-2u^*} = \mu(s)\, k'(s),
\end{equation}
which justifies the last equality in \eqref{chain_for_J}. 
\end{proof}

\section{Constrained problem}\label{section_constrained_problem}

As we have already pointed out, one of the main features of the results of previous sections is that they provide an explicit characterization of the optimal control process and stopping time, even though the problem under consideration is posed on an infinite time horizon and without discounting. It turns out that similarly explicit characterizations can be obtained in an analogous problem with extra constraints on the possible duration of the game. More precisely, we consider the minimization problem \eqref{value_function}, but impose the additional restriction that the expectation of available stopping times is bounded by some non-negative constant $\alpha$. Such an assumption aims to model real-life situations in which one cannot allow the system to run forever but, at the same time, does not want to implement a fixed time-horizon restriction, due to some uncertainty of planning.

Before discussing the constrained version of our problem, we want to mention some results available on optimal stopping problems with expectation constraints. Such problems were introduced by Kennedy in his seminal work \cite{Kennedy82}, where it was shown that, under appropriate assumptions, constrained problems can be reduced to unconstrained problems of optimal stopping but with the presence of a linear cost per unit of elapsed time. Surprisingly, despite their applicability and mathematical attractiveness, such problems did not receive much attention in the literature afterward. However, in recent years, a number of interesting results appeared in such a context. We refer the reader to the works of Ankirchner et al. \cite{AnkKazTanKleKru19}, \cite{AnkKleKru19}, Christensen et al. \cite{ChrKleSch23}, Bayraktar \& Yao \cite{BayYao20}, and to the very recent work \cite{BayYao23} by the same authors, where a constrained version of a problem which involves both stochastic control and optimal stopping is considered. We also refer the reader to the works of Bayraktar \& Miller \cite{BayMil19} and Källblad \cite{kallblad22}, where, in the context of optimal stopping, constraints are imposed on the distribution of the available stopping times.

We are now ready to formulate and study the constrained version of the problem \eqref{value_function}. We will follow the approach of Kennedy \cite{Kennedy82}.

\subsection{Description of the model}

We fix a parameter $\alpha \ge 0$ and place ourselves in the setup of Section \ref{section_model}. Without loss of generality, we focus attention again on the case $x \ge 0$ and take the notation from Section \ref{section_model} with the following necessary changes.

For a given ``state process'' $X(\cdot)$ and corresponding filtration $\mathbb{F}^X$, we denote by $\mathcal{T}^X_\alpha$ the collection of all $\mathbb{F}^X$--stopping times $\tau$ that satisfy $\ex[\tau] \le \alpha$. Now, the controller solves the following constrained optimization problem. \textit{For a given initial position $x \in \mathbb{R}$, find an element $X^\dag(\cdot) \in \mathcal{A}(x)$ and an associated  stopping time $\tau^\dag \in \mathcal{T}^{X^\dag}_\alpha$, such that the pair $(X^\dag(\cdot), \, \tau^\dag)$ minimizes the total expected cost}
\begin{equation}\label{total_cost_constrained}
    J(x, X(\cdot), \tau, c) \coloneqq \mathbb{E}\left[k(X(\tau))+\int_{0}^{\tau} \psi(u(t)) \, dt + c \tau\right]
\end{equation}
over all $X(\cdot) \in \mathcal{A}(x)$ and $\tau \in \mathcal{T}^X_\alpha$. 
Here, we take $c \ge 0$ and impose the same assumptions on the functions $k(\cdot)$ and $\psi(\cdot)$ as in Section \ref{section_model}.

We denote the value function of this problem by
\begin{equation}\label{value_function_constrained}
    V_\alpha(x, c) \coloneqq \inf\limits_{\substack{X(\cdot) \in \mathcal{A}(x) \\ \tau \in \mathcal{T}^X_\alpha}} J(x, X(\cdot), \tau, c).
\end{equation}

\begin{remark}
    Note that, unlike in the notation of Section \ref{section_model}, we include here the parameter $c$ to the variables of the expected payoff function \eqref{total_cost_constrained} and the value function \eqref{value_function_constrained}. It is clear that these functions indeed depend on the parameter $c$. The convenience, indeed the necessity, of such inclusion, will be clear soon.
\end{remark}

\begin{remark}
    We also note that the original problem \eqref{value_function} can be embedded into the one above by setting $\alpha = \infty$. This way, we have
    \begin{equation}
        V(x) = V_\infty(x, c) = \lim_{\alpha \to \infty} \downarrow V_\alpha(x, c).
    \end{equation}
\end{remark}

\subsection{Dual problem}

To solve the problem \eqref{value_function_constrained}, it will be convenient to use duality arguments, i.e., to treat \eqref{value_function_constrained} as a constrained \textit{primal} problem, and introduce a \textit{dual} unconstrained problem as follows.

For any initial position $x \ge 0$,  Lagrange multiplier $\lambda \ge 0$, admissible $X(\cdot) \in \mathcal{A}(x)$, and stopping time $\tau \in \mathcal{T}^X$, consider the Lagrangian $J(x, X(\cdot), \tau, c + \lambda) - \lambda \alpha$. Observe that we have
\begin{equation}\label{sup_lagrangian}
    \sup_{\lambda \ge 0}
    \Big(
        J(x, X(\cdot), \tau, c + \lambda) - \lambda \alpha 
    \Big) =
    \begin{cases}
        \infty, & \mathbb{E}[\tau] > \alpha, \\
        J(x, X(\cdot), \tau, c), & \mathbb{E}[\tau] \le \alpha,
    \end{cases}
\end{equation}
and, consequently, we obtain
\begin{equation}\label{duality_gap}
    \begin{split}
        V_\alpha(x, c)
        &=
        \inf\limits_{\substack{X(\cdot) \in \mathcal{A}(x) \\ \tau \in \mathcal{T}^X}} 
        \left(
        \sup_{\lambda \ge 0}
        \Big(
            J(x, X(\cdot), \tau, c + \lambda) - \lambda \alpha 
        \Big)
        \right)
        \\&\ge
        \sup_{\lambda \ge 0}
        \left(
        \inf\limits_{\substack{X(\cdot) \in \mathcal{A}(x) \\ \tau \in \mathcal{T}^X}} 
        \Big(
            J(x, X(\cdot), \tau, c + \lambda) - \lambda \alpha 
        \Big)
        \right)
        \eqqcolon
        \Psi_\alpha(x, c),
    \end{split}
\end{equation}
because $\inf \sup$ always dominates $\sup \inf$. From now onwards, we call the optimization problem with value $V_\alpha(x, c)$ \textit{primal}, and the optimization problem with value $\Psi_\alpha(x, c)$ \textit{dual}. 

\begin{remark}
    Note that the problem with value $\Psi_\alpha(x, c)$ in \eqref{duality_gap} is the problem of maximizing the function
    \begin{equation}
        f_\alpha(x, c, \lambda) 
        \coloneqq 
        \inf\limits_{\substack{X(\cdot) \in \mathcal{A}(x) \\ \tau \in \mathcal{T}^X}} 
        \Big(
            J(x, X(\cdot), \tau, c + \lambda) - \lambda \alpha 
        \Big)
        =
        V_\infty(x, c + \lambda) - \lambda \alpha
    \end{equation}
    over $\lambda \ge 0$, i.e., 
    \begin{equation}\label{remark_conv_conj}
        \Psi_\alpha(x, c) = \sup_{\lambda \ge 0} \Big(V_\infty(x, c + \lambda) - \lambda \alpha \Big).
    \end{equation}
    We recall at this point the notion of \textit{convex conjugate} for a given function $g: [0, \infty) \to \mathbb{R}$:
    \begin{equation}
        g^*(\beta) \coloneqq \sup_{\lambda \ge 0} \Big(g(\lambda) - \lambda \beta\Big), \quad \beta \ge 0.
    \end{equation}
    It is clear now from \eqref{duality_gap}, \eqref{remark_conv_conj} that $\alpha \mapsto \Psi_\alpha(x, c)$ is the convex conjugate of the function $V_\infty(x, c + \cdot)$.
\end{remark}

To solve the constrained problem, it is sufficient to show that there is no so-called ``duality gap'', meaning that the inequality in \eqref{duality_gap} is actually equality; and then to compute the function $\Psi_\alpha(x, c)$, which we can do since we have 
already solved the unconstrained problem. We first establish sufficient conditions for the absence of a duality gap in \eqref{duality_gap}, and then, in the next subsection, solve the constrained problem by checking that these conditions are satisfied.

The following result is an adaptation of Theorem 4 in \cite{Kennedy82} to our setting.

\begin{Prop}\label{prop_duality_gap}
    Fix the parameters $x, c$ and $\alpha$.
    Suppose there exist a stopping time $\tau^\dag$ and a constant $\bar{\lambda} \ge 0$, such that the inequality $\ex[\tau^\dag] \le \alpha$ and the ``complementary slackness'' condition $\bar{\lambda}(\alpha - \ex[\tau^\dag]) = 0$ hold. 
    If
    \begin{equation}\label{prop_techn_equality}
        \inf\limits_{\substack{X(\cdot) \in \mathcal{A}(x) \\ \tau \in \mathcal{T}^X}}
            J\left(x, X(\cdot), \tau, c + \bar{\lambda}\right)
        =
        J\left(x, X^\dag(\cdot), \tau^\dag, c + \bar{\lambda}\right)
    \end{equation}
    also holds for some controlled process $X^\dag(\cdot) \in \mathcal{A}(x)$, i.e., if the pair $\left(X^\dag(\cdot), \tau^\dag\right)$ is optimal for the unconstrained problem of \eqref{value_function} with cost $c + \bar{\lambda}$ per unit of time, then $\left(X^\dag(\cdot), \tau^\dag\right)$ is optimal for the constrained problem \eqref{value_function_constrained}. Moreover, we have then the equality $\Psi_\alpha(x, c) = V_\alpha(x, c)$, meaning that there is no ``duality gap'' in \eqref{duality_gap}.
\end{Prop}

\begin{proof}
    First, we show that, under the assumptions of the proposition, the pair $(X^\dag(\cdot), \tau^\dag)$ is indeed optimal for the constrained problem \eqref{value_function_constrained}. Consider any admissible pair $(X(\cdot), \tau)$ for the constrained problem. Then the following chain of inequalities
    \begin{equation}
        \begin{split}
            J\left(x, X(\cdot), \tau, c\right)
            &\ge 
            J\left(x, X(\cdot), \tau, c + \bar{\lambda}\right) - \bar{\lambda} \alpha
            \\&\ge
            \inf\limits_{\substack{X(\cdot) \in \mathcal{A}(x) \\ \tau \in \mathcal{T}^X}} 
            J\left(x, X(\cdot), \tau, c + \bar{\lambda}\right) - \bar{\lambda} \alpha
            \\&=
            J\left(x, X^\dag(\cdot), \tau^\dag, c +\bar{\lambda}\right) - \bar{\lambda} \alpha
            \\&=
            J\left(x, X^\dag(\cdot), \tau^\dag, c\right)
        \end{split}
    \end{equation}
    proves the optimality of $(X^\dag(\cdot), \tau^\dag)$.
    Here, the first inequality follows from the fact that $\bar{\lambda} \ge 0$ and $\ex[\tau] \le \alpha$, the first equality is a consequence of the assumption \eqref{prop_techn_equality}, and the second equality is a consequence of the assumption $\bar{\lambda}(\alpha - \ex[\tau^\dag]) = 0$.

    The equality $\Psi_\alpha(x, c) = V_\alpha(x, c)$ is a consequence of the inequality \eqref{duality_gap} and of the reverse inequality
    \begin{equation}
    \begin{split}
        \Psi_\alpha(x, c) 
        &= 
        \sup_{\lambda \ge 0}
        \left(
        \inf\limits_{\substack{X(\cdot) \in \mathcal{A}(x) \\ \tau \in \mathcal{T}^X}} \Big(
            J(x, X(\cdot), \tau, c + \lambda) - \lambda \alpha 
        \Big)
        \right)
        \\&\ge
        \inf\limits_{\substack{X(\cdot) \in \mathcal{A}(x) \\ \tau \in \mathcal{T}^X}} \Big(
            J(x, X(\cdot), \tau, c + \bar{\lambda}) - \bar{\lambda} \alpha 
        \Big)
        \\&=
        J(x, X^\dag(\cdot), \tau^\dag, c + \bar{\lambda}) - \bar{\lambda} \alpha
        \\&=
        J(x, X^\dag(\cdot), \tau^\dag, c)
        \ge
        V_\alpha(x, c),
    \end{split}
    \end{equation}
    where the third equality follows again by the choice of $\bar{\lambda}$ and $\tau^\dag$.
\end{proof}

\subsection{Solution of the constrained problem}

Recall the unconstrained optimal stopping problem \eqref{value_function} and denote, for the purposes of this section, its value function by $V(x, c)$. Recall also from Theorem \ref{main_result} that the optimal stopping boundary for this problem is the unique solution $s_c$ of the equation \eqref{equation_for_stopping_point} (if the solution exists, and $\infty$ if not). Here the subscript highlights the dependence of the stopping boundary on the parameter $c$ of the problem. In other words, in the notation of \eqref{def_zeta_function}, we have $s_c = \zeta^{-1}(-c)$ with the convention $s_c = \infty$ if $-c \notin \text{Im}(\zeta)$.

Recall the expression for the optimal control process \eqref{optimal_control}, which also depends on the parameter $c$ and will be denoted from now onwards by
\begin{equation}\label{optimal_control_2}
    u^*_c(t) = Q(s_c) = \left(\psi'\right)^{-1}\left(-\mu(s_c)\, k'(s_c) \right) = \upsilon(c), \quad \forall \,\, 0 \le t < \infty,
\end{equation}
where we define the function $\upsilon: [0, \infty) \to (-\infty, 0]$ by 
\begin{equation}\label{def_upsilon_function}
    \upsilon(z) \coloneqq Q(\zeta^{-1}(-z)) = \left(\psi'\right)^{-1}\left(-\mu\left(\zeta^{-1}(-z)\right) k'\left(\zeta^{-1}(-z)\right)\right).
\end{equation}

Finally, recall the expression \eqref{optimal_stopping_time} for the optimal stopping time, which we will denote here by $\tau^*_{x, c}$; as well as the expression \eqref{stop_time_expectation_formula} for the expectation of $\tau^*_{x, c}$, which in the new notation becomes
\begin{equation}\label{stop_time_expectation_formula_2}
    \mathbb{E}\left[\tau^*_{x, c}\right] 
    =
    \frac{2}{A - 2u^*_c}\int_{s_c}^x \frac{dy}{\mu(y)} \vee 0
    = 
    \frac{2}{A - 2\upsilon(c)}\int_{\zeta^{-1}(-c)}^x \frac{dy}{\mu(y)} \vee 0.
\end{equation}

We note at this point that the functions $\zeta(\cdot)$ in \eqref{def_zeta_function} and $\upsilon(\cdot)$ in \eqref{def_upsilon_function} are strictly decreasing, due to the monotonicity of the functions $\eta(\cdot), \mu(\cdot)$ and the convexity of the functions $\psi(\cdot), k(\cdot)$. Moreover, this implies that the expectation of $\tau^*_{x, c}$ in \eqref{stop_time_expectation_formula_2} is also a decreasing function of $c$.

\begin{remark}
    The purpose of the new notation is to emphasize the dependence of various expressions on the parameter $c$, and, moreover, to stress that these expressions can be represented as functions of $c$.
\end{remark}

We are now ready to state the result about the constrained problem.

\begin{Th}
    For a fixed $\alpha \ge 0$ and a pair $(x, c) \in [0, \infty) \times [0, \infty)$, consider the constrained problem $V_\alpha(x, c)$ of \eqref{value_function_constrained}. 
    For each $\lambda \ge 0$, consider the unconstrained problem $V(x, c + \lambda)$ of \eqref{value_function}, and let 
    $\left(X^*_{x, \, c + \lambda}(\cdot), \tau^*_{x, \, c + \lambda}\right)$ 
    be the optimal pair in the corresponding problem, given by the expressions \eqref{optimal_control} and \eqref{optimal_stopping_time}. 
    Denote by 
    \begin{equation}\label{lagrange_multiplier}
        \widehat{\lambda}_{x, c} \coloneqq \inf \left\{\lambda \in [0, \infty): \ex\left[\tau^*_{x, \, c + \lambda}\right] \le \alpha \right\}
    \end{equation}
    the smallest value of the Lagrange multiplier, for which the expectation of the optimal stopping time for the unconstrained problem with cost parameter $c + 
    \lambda$ is at most $\alpha$. 
    Then 
    \begin{equation}
        V_\alpha(x, c) = V\left(x, c + \widehat{\lambda}_{x, c}\right) - \widehat{\lambda}_{x, c} \cdot \alpha,
    \end{equation}
    and the pair $\left(X^*_{x, \, c + \widehat{\lambda}_{x, c}}(\cdot), \tau^*_{x, \, c + \widehat{\lambda}_{x, c}}\right)$ is optimal for the constrained problem.
\end{Th}


\begin{proof}
Fix $\alpha \ge 0, \, x \ge 0, \, c \ge 0$. First, we check that the Lagrange multiplier $\widehat{\lambda}_{x, c}$ in \eqref{lagrange_multiplier} is well defined, meaning that the set $\left\{\lambda \in [0, \infty): \ex\left[\tau^*_{x, \, c + \lambda}\right] \le \alpha \right\}$ is not empty. This follows from the fact that the expectation of the stopping time in \eqref{stop_time_expectation_formula_2} is a decreasing function of the parameter $c$ and tends to zero as $c$ tends to infinity. Thus, for any $\alpha \ge 0$ and a fixed starting position $x \ge 0$, the expectation of the corresponding stopping time $\tau^*_{x, c + \lambda}$ will be no greater than $\alpha$ for large enough $\lambda$.

Now we note that, with the functions $\zeta(\cdot)$ and $\upsilon(\cdot)$ defined in \eqref{def_zeta_function} and \eqref{def_upsilon_function}, respectively, $\widehat{\lambda}_{x, c}$ can alternatively be characterized as the unique $\lambda > 0$ that solves the equation
\begin{equation}
    \frac{2}{A - 2\upsilon(c + \lambda)}\int_{\zeta^{-1}(-c-\lambda)}^x \frac{dy}{\mu(y)} = \alpha,
\end{equation}
if a solution exists, and $\widehat{\lambda}_{x, c} = 0$ if the above equation doesn't have a solution in $(0, \infty)$. This, in turn, implies that, if we denote by $\left(X^*_{x, \, c + \widehat{\lambda}_{x, c}}(\cdot), \tau^*_{x, \, c + \widehat{\lambda}_{x, c}}\right)$ the optimal pair for the unconstrained problem with value $V\left(x, c + \widehat{\lambda}_{x, c}\right)$, we will have 
\begin{equation}
    \widehat{\lambda}_{x, c} \cdot \left(\alpha - \ex\left[\tau^*_{x, \, c + \widehat{\lambda}_{x, c}} \right] \right) = 0
\end{equation}
and, moreover, the triple $\left(\widehat{\lambda}_{x, c}, X^*_{x, \, c + \widehat{\lambda}_{x, c}}(\cdot), \tau^*_{x, \, c + \widehat{\lambda}_{x, c}} \right)$ will satisfy \eqref{prop_techn_equality}. As a result, Proposition \ref{prop_duality_gap} implies that the pair 
$\left(X^*_{x, \, c + \widehat{\lambda}_{x, c}}(\cdot), \tau^*_{x, \, c + \widehat{\lambda}_{x, c}}\right)$ is optimal for the constrained problem with value $V_\alpha(x, c)$.

It only remains to note that Proposition \ref{prop_duality_gap} also implies 
\begin{equation}
    V_\alpha(x, c) 
    =
    \Psi_\alpha(x, c) 
    =
    \sup_{\lambda \ge 0}
    \big(
        V(x, c + \lambda) - \lambda \alpha
    \big),
\end{equation}
as in \eqref{remark_conv_conj}, so we have 
\begin{equation}
    V_\alpha(x, c) 
    = 
    \sup_{\lambda \ge 0}
    \big(
        V(x, c + \lambda) - \lambda \alpha
    \big)
    =
    V\left(x, c + \widehat{\lambda}_{x, c}\right) - \widehat{\lambda}_{x, c} \cdot \alpha
\end{equation}
as a consequence of the so-called ``envelope equation''
\begin{equation}\label{envelope_equation}
    \frac{\partial}{\partial c} V(x, c) = \ex\left[ \tau^*_{x, c}\right].
\end{equation}
The above identity \eqref{envelope_equation} is a simple algebraic consequence of the explicit form of the value function \eqref{value_function_solution} and the explicit form of the expectation \eqref{stop_time_expectation_formula_2}; we omit the technical but straightforward details.
\end{proof}

\section{Open Questions}\label{section_open_questions}

We conclude by formulating three directions for further research of related problems, which we find interesting.

The first direction would be to explore game versions of the problem formulated in Section \ref{section_model}. More precisely, instead of assuming that the same person controls the system and chooses stopping time, one can consider the following setup. There exist two players: a controller, who chooses an admissible control and tries to minimize the expected payoff \eqref{value_function}, and a stopper, who picks a stopping time and tries to maximize the expected payoff \eqref{value_function}. The goal of the problem is again to prove the existence of a Nash equilibrium and, consequently, find all such equilibria. 

The second direction would be to follow the footsteps of Section \ref{section_constrained_problem}, but to impose different time ``restrictions'' on our problem. For instance, one can consider the same problem over a finite time horizon. However, in this case, we believe the solution will not be so explicit; in particular, we expect the optimal stopping time to be a hitting time of an appropriate moving boundary. Another option would be to keep an infinite time horizon but add a discounting factor (see, for instance, the interesting work of Bayraktar \& Young \cite{BayYou11}). 

Finally, the third and most ambitious direction would be to add extra uncertainty to our problem. Namely, one can try to generalize the previous direction by considering a control problem with discretionary stopping and random maturity, in conjunction with the work of Bayraktar \& Yao \cite{BayYao17}, where the control of the drift is essentially given by non-linear expectation. Or one can place the problem in the context of filtering theory.

\vspace{15pt}
\noindent
\textbf{Acknowledgments}

We thank the anonymous referee for valuable comments that helped to improve the exposition of the paper.
We also gratefully acknowledge support from the National Science Foundation under grant NSF-DMS-20-04997.

\begin{appendices}

\section{Controlling the variance}

In this section, we discuss briefly a possible extension of the results in Section \ref{section_main_results}, in which
the control affects both the drift and the dispersion of the diffusion process, always in the presence of discretionary stopping. More specifically, we consider the same problem as formulated in Section \ref{section_model}, but now the controlled diffusion process $X(\cdot)$ satisfies 
\begin{equation}\label{controlled_sde_with_variance}
    X(t) = x + \int_{0}^{t} u(r) \,dr + \int_{0}^{t} u(r)\sigma (X(r)) \, dW(r), \quad 0 \le t < \infty.
\end{equation}
The goal of the problem is again to minimize the total expected cost
\begin{equation}\label{total_cost_extension_section}
    J(x, X(\cdot), \tau) \coloneqq \mathbb{E}\left[k(X(\tau))+\int_{0}^{\tau} \psi(u(t)) \, dt + c \tau\right],
\end{equation}
where we use the same notation and impose the same assumptions as in Section \ref{section_model}, but with the following necessary change in the assumption \ref{ass_3}:
\begin{enumerate}[label = \textbf{(A\arabic*)$'$}]
\setcounter{enumi}{2}
    \item $\int_{0}^{t}\big[|u(r)|+\big(u(r)\, \sigma(X(r))\big)^{2}\big] \, dr < \infty$ almost surely, for all $0 \leq t < \infty$. \label{ass_3_ext}
\end{enumerate}
For technical reasons, we also need to impose the following additional assumption:
\begin{enumerate}[label = \textbf{(A\arabic*)}]
\setcounter{enumi}{7}
    \item Either we restrict ourselves to bounded control processes $u(\cdot)$, or we assume that the running cost function $\psi(\cdot)$ is super-quadratic, meaning that $\liminf_{x \to \infty} \left(\psi(x)/x^2\right) > 0$. \label{ass_8_ext}
\end{enumerate}

It turns out, that in this case the results of Theorem \ref{main_result} for the drift coefficient function $\mu(\cdot) \equiv 1$ can be transferred to the current setup without any changes. This follows from the simple observation that the value function defined by \eqref{value_function_solution} is linear and thus still satisfies the adjusted variational inequalities 
\begin{enumerate}[label = \textbf{(\roman*)$'$}]
    \item $k(x) - V(x) \ge 0$; \label{var_ineq_1_ext}
    \item $\min\limits_{u \in \mathbb{R}}\Big[\frac{1}{2}V''(x)\sigma^2(x)u^2 + u \,\mu(x)\, V'(x) + \psi(u) \Big] + c \ge 0$; \label{var_ineq_2_ext}
    \item $\Big( k(x) - V(x)\Big)\left( \min\limits_{u \in \mathbb{R}}\Big[\frac{1}{2}V''(x)\sigma^2(x)u^2 + u \,\mu(x)\, V'(x) + \psi(u) \Big] + c
    \right) = 0$. \label{var_ineq_3_ext}
\end{enumerate}
As a result, the proof of the Theorem \ref{main_result} can be repeated without major changes because of the following two facts. First, the assumption \ref{ass_8_ext} implies that the stochastic integral in the adjusted for the variance term representation \eqref{Ito} is again a square-integrable martingale. Hence, after the variance term adjustment, \eqref{value_function_as_lower_bound} still holds, which proves that $V(x)$ is indeed a lower bound on the achievable cost when the starting position is $X(0) = x$. Secondly, Proposition \ref{stopping_time_expectation} is also still valid, since the claim \eqref{stop_time_expectation_formula} doesn't depend on the variance coefficient of the diffusion process, and multiplication of $\sigma(\cdot)$ by a constant doesn't change the arguments in the proof.

\section{Proof of Proposition \ref{stopping_time_expectation}}\label{appendix_proof}

\begin{proof_prop_2}
We will derive the claim  \eqref{stop_time_expectation_formula} using a general expression for the expectation of one-sided hitting times for diffusion processes. 

First, we introduce the following notions. Fix $d > s$. We define the so-called \textit{scale function} of the diffusion \eqref{proposition_diffusion} by 
\begin{equation}\label{scale_function}
    p(x) \coloneqq \int_{d}^{x} \exp \left\{-2 \int_{d}^{\xi} \frac{u \, \mu(\zeta) d \zeta}{\sigma^{2}(\zeta)}\right\} d \xi, \quad x \in [s, \infty),
\end{equation}
which satisfies
\begin{equation}\label{PolSig_condition_a}
    p(\infty) = \int_{d}^{\infty} \exp \left\{-2 \int_{d}^{\xi} \frac{u \, \mu(\zeta) d \zeta}{\sigma^{2}(\zeta)}\right\} d \xi \ge \int_{d}^{\infty} \exp(0) \, d\xi = \infty.
\end{equation}
Here, the inequality follows from the non-negativity of $\mu(\cdot)$ and the fact that $u < 0$ by the assumption of the proposition.

We define also the so-called \textit{speed measure} of the diffusion \eqref{proposition_diffusion} by
\begin{equation}\label{speed_measure}
    m(d x) \coloneqq \frac{2 d x}{p^{\prime}(x) \sigma^{2}(x)}, \quad x \in [s, \infty),
\end{equation}
which, in its turn, satisfies 
\begin{equation}\label{PolSig_condition_b}
    m((s, \infty)) < \infty.
\end{equation}
To check \eqref{PolSig_condition_b}, we consider the cases $A = 0$ and $A > 0$ separately. If $A = 0$, we have already shown (see the discussion right before the bound \eqref{second_upper_bound_b}) that the function $\mu(\cdot)$ satisfies $\mu(x) \equiv \mu$ for all $x > 0$ and some positive real constant $\mu$. Therefore, we obtain \eqref{PolSig_condition_b} by the following computation:
\begin{align}\label{full_measure_zero_case}
        m((s, \infty)) &= \int_s^\infty 2\, \sigma^{-2}(x) \, \exp\left\{2 \int_{d}^{x} \frac{u \mu \, d \zeta}{\sigma^{2}(\zeta)}\right\}\, dx
        =
        \frac{1}{u\mu}\exp\left\{2u\mu \int_{d}^{x} \frac{ d \zeta}{\sigma^{2}(\zeta)}\right\} \Bigg|_s^\infty
        \\ &= \frac{1}{u\mu}\exp\left\{2u\mu \int_{d}^{\infty} \frac{ d \zeta}{\sigma^{2}(\zeta)}\right\} - \frac{1}{u\mu}\exp\left\{2u\mu \int_{d}^{s} \frac{ d \zeta}{\sigma^{2}(\zeta)}\right\} 
        = 
        - \frac{1}{u\mu}\exp\left\{2u\mu \int_{d}^{s} \frac{ d \zeta}{\sigma^{2}(\zeta)}\right\} < \infty,
\end{align}
where the last equality follows from the assumption \ref{unf_assum} and the fact that $\mu > 0$, while $u < 0$.

To obtain \eqref{PolSig_condition_b} in the case $A > 0$, we first rewrite the scale function and speed measure in a more convenient way. Namely, using \ref{ass_4}, we get
\begin{equation}
    \begin{split}
        p(x) &=
        \int_{d}^{x} \exp \left\{-2 \int_{d}^{\xi} \frac{u \, \mu(\zeta) d \zeta}{\sigma^{2}(\zeta)}\right\} d \xi
        =
        \int_{d}^{x} \exp \left\{-2 \int_{d}^{\xi} \frac{u\, \mu'(\zeta) \, d \zeta}{A\mu(\zeta)}\right\} d \xi
        \\ &=
        \int_{d}^{x} \exp \left\{-\frac{2u}{A} \log(\mu(\zeta) )\Big|_{d}^{\xi}\right\} d \xi
        =
        \int_{d}^{x} \exp \left\{ 
        \log\left(\left(\frac{\mu(\xi)}{\mu(d)} \right)^{-\frac{2u}{A}}\right)
        \right\} d \xi
        \\ &=
        \int_{d}^{x} \left(\frac{\mu(\xi)}{\mu(d)} \right)^{-\frac{2u}{A}} d \xi
        =
        \big(\mu(d)\big)^{\frac{2u}{A}} \cdot \int_{d}^{x} \big(\mu(\xi)\big)^{-\frac{2u}{A}} d \xi,
    \end{split}
\end{equation}
thus also
\begin{equation}
    p'(x) = \big(\mu(d)\big)^{\frac{2u}{A}} \cdot \big(\mu(x)\big)^{-\frac{2u}{A}},
\end{equation}
which leads to the expression 
\begin{equation}\label{scale_measure_special_form}
        m(dx) = 
        \frac{2dx}{p'(x)\sigma^2(x)}
        =
        \frac{2 \mu'(x) dx}{A p'(x)\mu^2(x)}
        =
        \frac{2}{A}\big(\mu(d)\big)^{-\frac{2u}{A}} \cdot \mu'(x) \cdot \big(\mu(x)\big) ^{\frac{2u}{A} - 2} \, dx
\end{equation}
for the speed measure of \eqref{speed_measure}.
Using the above representation of the speed measure, we get
\begin{equation}\label{full_measure_nonzero_case}
    \begin{split}
        m((s, \infty)) &= \int_s^\infty \frac{2}{A}\big(\mu(d)\big)^{-\frac{2u}{A}} \cdot \mu'(x) \cdot \big(\mu(x)\big) ^{\frac{2u}{A} - 2} \, dx
        \\ &=
        \frac{2}{2u - A} \big(\mu(d)\big)^{-\frac{2u}{A}} \, \big(\mu(x)\big) ^{\frac{2u}{A} - 1} \bigg|_s^\infty 
        \\&=
        -\frac{2}{2u - A} \big(\mu(d)\big)^{-\frac{2u}{A}} \, \big(\mu(s)\big) ^{\frac{2u}{A} - 1} < \infty,
    \end{split}
\end{equation}
where the third equality follows from the fact that in the case $A > 0$, we have $\mu(\infty) = \infty$ by the assumption \ref{unf_assum}, while $2u - A < 0$ as $u < 0$ and $A > 0$.

Once the properties \eqref{PolSig_condition_a} of the scale function and \eqref{PolSig_condition_b} of the speed measure have been established, we can use the results of Pollack \& Siegmund \cite{PolSig} (see also pp. 352-353 in \cite{BMSC}), which allow us to compute the expectation of the one-sided hitting time \eqref{one_sided_hitting_time} by the formula
\begin{equation}\label{stopping_time_bmsc_formula}
    \mathbb{E}[\tau] = -\int_{s}^{x}(p(x)-p(y)) m(d y)+(p(x)-p(s)) \cdot m((s, \infty)).
\end{equation}
It remains to show that the above expression is indeed equal to the right-hand side of \eqref{stop_time_expectation_formula}. 

For convenience of analysis and exposition, we fix $x \in [s, \infty)$ and introduce the functions
\begin{equation}\label{functions_h_g}
    h(y) \coloneqq p(x) - p(y) \quad \text{ and } \quad g(y) \coloneqq m((s, y)),
\end{equation}
so that the expectation in \eqref{stopping_time_bmsc_formula} can be written as
\begin{equation}
    \ex[\tau] = -\int_s^x h(y) \, dg(y) + h(s)g(\infty).
\end{equation}
Using integration by parts, we obtain
\begin{equation}\label{expectation_explicit}
    \ex[\tau] = -h(y) g(y) \Big|_s^x + \int_s^x h'(y)g(y)dy + h(s)g(\infty) = \int_s^x h'(y)g(y)dy + h(s)g(\infty),
\end{equation}
where the second equality follows from the consequence $h(x) = g(s) = 0$ of the definitions of these functions.

To proceed further, we compute the integral in the above expression. We need again to consider two cases. If $A = 0$, we obtain
\begin{equation}\label{expectation_explicit_1}
    \begin{split}
        \int_s^x h'(y)g(y)dy 
        &=
        -\int_s^x \exp \left\{-2 \int_{d}^{y}\frac{u \mu \, d \zeta}{\sigma^{2}(\zeta)}\right\} \cdot \left( \frac{1}{u\mu}\exp\left\{2u\mu \int_{d}^{x} \frac{ d \zeta}{\sigma^{2}(\zeta)}\right\} \Bigg|_s^y\right) dy
        \\ &=
        -\int_s^x \frac{dy}{u\mu} - \int_s^x h'(y) \cdot \frac{1}{u\mu}\exp\left\{2u\mu \int_{d}^{s} \frac{ d \zeta}{\sigma^{2}(\zeta)}\right\} \, dy
        \\ &= 
        -\frac{x-s}{u\mu} - h(s)g(\infty),
    \end{split}
\end{equation}
where the last equality follows from the computation \eqref{full_measure_zero_case} and the fact that $h(x) = 0$. 

If, on the other hand, $A > 0$, we get 
\begin{equation}\label{expectation_explicit_2}
    \begin{split}
        \int_s^x h'(y)g(y)dy 
        &= 
        -\int_s^x \left(\big(\mu(d)\big)^{\frac{2u}{A}} \cdot \big(\mu(y)\big)^{-\frac{2u}{A}}\right) \left(\frac{2}{2u - A} \big(\mu(d)\big)^{-\frac{2u}{A}} \, \big(\mu(x)\big) ^{\frac{2u}{A} - 1} \bigg|_s^y \right) dy
        \\&= 
        \frac{2}{A-2u} \int_s^x \frac{dy}{\mu(y)} - \int_s^x h'(y) \cdot \frac{2}{2u - A} \big(\mu(d)\big)^{-\frac{2u}{A}} \, \big(\mu(s)\big) ^{\frac{2u}{A} - 1} dy 
        \\&=
        \frac{2}{A-2u} \int_s^x \frac{dy}{\mu(y)} - h(s)g(\infty),
    \end{split}
\end{equation}
where the last equality follows from \eqref{full_measure_nonzero_case}, \eqref{functions_h_g}, and the fact that $h(x) = 0$.

Combining now \eqref{expectation_explicit}, \eqref{expectation_explicit_1} and \eqref{expectation_explicit_2}, and observing that in the case $A = 0$ we have 
\begin{equation}
    \frac{2}{A - 2u}\int_s^x \frac{dy}{\mu(y)} = -\frac{x-s}{u\mu}, \quad \quad \text{thus} \quad \quad
    \ex[\tau] = \frac{2}{A - 2u}\int_s^x \frac{dy}{\mu(y)},
\end{equation}
which justifies the claim \eqref{stop_time_expectation_formula}. \qed
\end{proof_prop_2}

\end{appendices}

\printbibliography

@book{BMSC,
    author = {Karatzas, I. and Shreve, S.},
    title = {Brownian Motion and Stochastic Calculus. {\normalfont Second Edition.}  },
    %volume = "2nd Edition,",
    year = {1991},
    publisher = {Springer, New York} 
}

@article{DavZer,
    author = {Davis, M.H.A. and Zervos, M.},
    title = "A Problem of Singular Stochastic Control with Discretionary Stopping",
    journal = "Ann. Appl. Probab.",
    volume = "4 (1),",
    pages = "226-240",
    year = "1994"
}

@article{KarOcoWanZer,
    author = {Karatzas, I. and Ocone, D. and Wang, H. and Zervos, M.},
    title = "Finite-Fuel Singular Control With Discretionary Stopping",
    journal = "Stochastics",
    volume = "71 (1-2),",
    pages = "1-50",
    year = "2000"
}

@incollection{OcoWeer08,
  title={A degenerate variance control problem with discretionary stopping},
  author={Ocone, D. and Weerasinghe, A.},
  booktitle={Markov Processes and Related Topics: A Festschrift for Thomas G. Kurtz},
  volume={4},
  pages={155--168},
  year={2008},
  publisher={Institute of Mathematical Statistics}
}

@article{PolSig,
    author = {Pollack, M. and Siegmund, D.},
    title = "A diffusion process and its applications to detecting a change in the drift of Brownian motion",
    journal = "Biometrika",
    volume = "72,",
    pages = "267-280",
    year = "1985"
}

@book{DubSav65,
    author = {Dubins, L.E. and Savage, L.J.},
    title = {How to Gamble If You Must: Inequalities for Stochastic Processes.},
    year = {1965},
    publisher = {McGraw-Hill, New York} 
}

@book{BenLio82,
    author = {Bensoussan, A. and Lions, J.L.},
    title = {Applications of Variational Inequalities in Stochastic Control.},
    year = {1982},
    publisher = {North-Holland, Amsterdam and American Elsevier, New York} 
}

@book{Krylov80,
    author = {Krylov, N.V.},
    title = {Controlled Diffusion Processes.},
    year = {1980},
    publisher = {Springer-Verlag, New York} 
}

@book{ElKaroui81,
    author = {El Karoui, N.},
    title = {Les Aspects Probabilistes du Contrôle Stochastique. Lecture Notes in Mathematics 876.},
    year = {1981},
    pages = "72-238",
    publisher = {Springer-Verlag, Berlin} 
}

@article{Benes92,
    author = {Beneš, V.E.},
    title = "Some combined control and stopping problems",
    journal = "Paper presented at the CRM Workshop on Stochastic Systems, Montréal",
    year = "1992"
}

@incollection{MaiSud96p,
  title={The gambler and the stopper},
  author={Maitra, A.P. and Sudderth, W.D.},
  booktitle={{\normalfont In} Statistics, Probability and Game Theory: Papers in Honor of David Blackwell {\normalfont(T.S. Ferguson, L.S. Shapley \& J.B. MacQueen, Editors.)},},
  publisher={\textit{IMS Lecture Notes-Monograph Series}, 30,},
  pages={191--208},
  year={1996}
}

@book{MaiSud96b,
    author = {Maitra, A.P. and Sudderth, W.D.},
    title = {Discrete Gambling and Stochastic Games.},
    year = {1996},
    publisher = {Springer, New York} 
}

@article{KarSud99,
    author={Karatzas, I. and Sudderth, W.D.},
    title={Control and stopping of a diffusion process on an interval},
    journal={Ann. Appl. Probab.},
    volume = "9,",
    pages={188--196},
    year={1999}
}

@article{KarSud01,
    author={Karatzas, I. and Sudderth, W.D.},
    title={The controller and stopper game for a linear diffusion},
    journal={Ann. Probab.},
    volume = "29,",
    pages={1111-1127},
    year={2001}
}

@incollection{KarSud06,
    author={Karatzas, I. and Sudderth, W.D.},
    title={Stochastic games of control and stopping for a linear diffusion},
    booktitle={{\normalfont In} Random Walk, Sequential Analysis and Related Topics: A Festschrift in Honor of Y.S. Chow {\normalfont(A. Hsiung, Zh. Ying \& C.H. Zhang, Editors.)},},
    publisher={\textit{World Scientific Publishers,}},
    pages={100-117},
    year={2006}
}

@article{KarWang01,
    author={Karatzas, I. and Wang, H.},
    title={Utility maximization with discretionary stopping},
    journal={SIAM J. Control Optim.},
    volume = "39,",
    pages={306-329},
    year={2001}
}

@article{KarWang00,
  author={Karatzas, I. and Wang, H.},
  title={A barrier option of American type},
  journal={Appl. Math. Optim.},
  volume = "42,",
  pages={259-279},
  year={2000}
}

@article{KarOco02,
  author={Karatzas, I. and Ocone, D.},
  title={A leavable, bounded-velocity stochastic control problem},
  journal={Stoch. Process. Their Appl.},
  volume = "99,",
  pages={31-51},
  year={2002}
}

@article{KamMor02,
  title={On a variational inequality associated with a stopping game combined with a control},
  author={Kamizono, K. and Morimoto, H.},
  journal={Stoch. Stoch. Rep.},
  volume={73 (1-2),},
  pages={99--123},
  year={2002},
  publisher={Taylor \& Francis}
}

@article{Morimoto03,
  title={Variational inequalities for combined control and stopping},
  author={Morimoto, H.},
  journal={SIAM J. Control Optim.},
  volume={42 (2),},
  pages={686--708},
  year={2003},
  publisher={SIAM Journal on Control \& Optimization,}
}

@article{CeciBas04,
  title={Mixed optimal stopping and stochastic control problems with semicontinuous final reward for diffusion processes,},
  author={Ceci, C. and Bassan, B.},
  journal={Stoch. Stoch. Rep.},
  volume={76 (4),},
  pages={323--337},
  year={2004},
  publisher={Taylor \& Francis}
}

@article{Weer06,
  title={A controller and a stopper game with degenerate variance control},
  author={Weerasinghe, A.},
  journal={Electron. Commun. Probab.},
  volume={11,},
  pages={89-99},
  year={2004},
}

@article{KarZamf06,
  title={Martingale approach to stochastic control with discretionary stopping},
  author={Karatzas, I. and Zamfirescu, I.M.},
  journal={Appl. Math. Optim.},
  volume={53,},
  pages={163--184},
  year={2006}
}

@article{KarZamf08,
  title={Martingale approach to stochastic differential games of control and stopping},
  author={Karatzas, I. and Zamfirescu, I.M.},
  journal={Ann. Probab.},
  volume={36,},
  pages={1495–1527},
  year={2008}
}

@article{HenHob08,
  title={An explicit solution for an optimal stopping/control problem which models an asset sale},
  author={Henderson, V. and Hobson, D.},
  journal={Ann. Appl. Probab.},
  volume={18 (5),},
  pages={1681–1705},
  year={2008}
}

@article{LeuSir09,
  title={Exponential hedging with optimal stopping and application to employee stock option valuation},
  author={Leung, T. and Sircar, R.},
  journal={SIAM J. Control Optim.},
  volume={48 (3),},
  pages={1422--1451},
  year={2009},
}

@article{KarKou98,
  title={Hedging American contingent claims with constrained portfolios},
  author={Karatzas, I. and Kou, S.G.},
  journal={Finance Stoch.},
  volume={3,},
  pages={215-258},
  year={1998},
}

@article{DeAngMil23,
  title={Dynamic Programming Principle for Classical and Singular Stochastic Control with Discretionary Stopping},
  author={De Angelis, T. and Milazzo, A.},
  journal={Appl. Math. Optim.},
  volume={88 (7),},
  %pages={215-258},
  year={2023},
}

@article{HerSimZer15,
  author = {Hernández-Hernández, D. and Simon, R.S. and Zervos, M.},
  title = {A zero-sum game between a singular stochastic controller and a discretionary stopper},
  journal = {Ann. Appl. Probab.},
  volume = {1 (25),},
  publisher = {Institute of Mathematical Statistics},
  pages = {46 -- 80},
  year = {2015},
}

@article{DumLeuTan21,
  author = {Dumitrescu, R. and Leutscher, M. and Tankov, P.},
  title = {Control and optimal stopping Mean Field Games: a linear programming approach},
  volume = {26},
  journal = {Electron. J. Probab.},
  pages = {1 -- 49},
  year = {2021},
}

@article{MusZar09,
  title={Portfolio choice under dynamic investment performance criteria},
  author={Musiela, M. and Zariphopoulou, T.},
  journal={Quant. Finance},
  volume={9 (2),},
  pages={161--170},
  year={2009}
}

@incollection{MusZar10,
  title={Stochastic partial differential equations and portfolio choice},
  author={Musiela, M. and Zariphopoulou, T.},
  booktitle={Contemporary Quantitative Finance: Essays in Honour of Eckhard Platen,},
  pages={195--216},
  year={2010},
  publisher={Springer}
}

@article{MusZar07,
  title={Investment and valuation under backward and forward dynamic exponential utilities in a stochastic factor model},
  author={Musiela, M. and Zariphopoulou, T.},
  journal={Advances in Mathematical Finance},
  pages={303--334},
  year={2007},
  publisher={Springer}
}

@article{Ber09,
  title={A characterization of forward utility functions},
  author={Berrier, F.P.Y.S., and Rogers, L.C.G. and Tehranchi, M.R.},
  journal={Preprint},
  year={2009}
}

@article{BayHua12,
  title={On the multi-dimensional controller and stopper games},
  author={Bayraktar, E. and Huang, Y.-L.},
  journal={SIAM J. Control Optim.},
  pages={1263–1297},
  year={2012}
}

@article{BayMil19,
  title={Distribution-Constrained Optimal Stopping},
  author={Bayraktar, E. and Miller, C.W.},
  journal={Math. Finance},
  volume={29 (1),},
  pages={368-406},
  year={2019}
}

@article{kallblad22,
  title={A dynamic programming approach to distribution-constrained optimal stopping},
  author={K{\"a}llblad, S.},
  journal={Ann. Appl. Probab.},
  volume={32 (3)},
  pages={1902--1928},
  year={2022},
  publisher={Institute of Mathematical Statistics}
}

@article{Kennedy82,
  title={On a constrained optimal stopping problem},
  author={Kennedy, D.P.},
  journal={J. Appl. Probab.},
  volume={19 (3),},
  pages={631--641},
  year={1982}
}

@article{AnkKazTanKleKru19,
  title={Stopping with expectation constraints: 3 points suffice},
  author={Ankirchner, S. and Kazi-Tani, N. and Klein, M. and Kruse, T.},
  journal={Electron. J. Probab.},
  volume={24,},
  pages={1--16},
  year={2019}
}

@article{AnkKleKru19,
  title={A Verification Theorem for Optimal Stopping Problems with Expectation Constraints},
  author={Ankirchner, S. and Klein, M. and Kruse, T.},
  journal={Appl. Math. Optim.},
  volume={79 (1),},
  pages={145--177},
  year={2019}
}

@article{ChrKleSch23,
  title={On the time consistent solution to optimal stopping problems with expectation constraint},
  author={Christensen, S. and Klein, M. and Schultz, B.},
  journal={Appl. Math. Optim.},
  volume={91 (1),},
  pages={11},
  year={2025}
}

@article{BayYao23,
  title={Stochastic control/stopping problem with expectation constraints},
  author={Bayraktar, E. and Yao, S.},
  journal={Stoch. Process. Their Appl.},
  volume={176,},
  pages={104430},
  year={2024}
}

@article{BayYao20,
  title={Optimal Stopping with Expectation Constraints},
  author={Bayraktar, E. and Yao, S.},
  journal={Ann. Appl. Probab.},
  volume={34 (1B),},
  pages={917-959},
  year={2024}
}

@article{BayYao17,
  title={Optimal Stopping with Random Maturity under Nonlinear Expectations},
  author={Bayraktar, E. and Yao, S.},
  journal={Stoch. Process. Their Appl.},
  volume={127 (8),},
  pages={2586--2629},
  year={2017}
}

@article{BayYou11,
  title={Proving the Regularity of the Minimal Probability of Ruin via a Game of Stopping and Control},
  author={Bayraktar, E. and Young, V.},
  journal={Finance Stoch.},
  volume={15 (4),},
  pages={785--818},
  year={2011}
}

\end{document}